\documentclass[review,onefignum,onetabnum]{siamart190516}

\usepackage{lipsum,bm}
\usepackage{amsfonts,wasysym}
\usepackage{graphicx}
\usepackage{epstopdf}
\usepackage{algorithmic}
\usepackage{amsopn}

\ifpdf
\DeclareGraphicsExtensions{.eps,.pdf,.png,.jpg}
\else
\DeclareGraphicsExtensions{.eps}
\fi

\newsiamremark{rem}{Remark}
\newsiamremark{defn}{Definition}
\newsiamthm{cor}{Corollary}
\newsiamthm{assum}{Assumption}
\newsiamthm{thm}{Theorem}
\newsiamthm{prop}{Proposition}
\newsiamthm{lem}{Lemma}

\newcommand{\diag}[1]{\mathrm{diag}(#1)}

\title{On the convergence conditions of Laplace importance sampling with randomized quasi-Monte Carlo\thanks{Submitted to the editors DATE.
		\funding{This work of the first author was funded by the National Science Foundation of China (No. 12071154)}}}

\author{Zhan Zheng\thanks{Corresponding author. Department of Mathematical Sciences, Tsinghua University, Beijing 100084, People's Republic of China (\email{zhengz15@mails.tsinghua.edu.cn}).}
\and Hejin Wang\thanks{Department of Mathematical Sciences, Tsinghua University, Beijing 100084, People's Republic of China (\email{wanghj20@mails.tsinghua.edu.cn}).}	
\and Xiaoqun Wang\thanks{Department of Mathematical Sciences, Tsinghua University, Beijing 100084, People's Republic of China (\email{wangxiaoqun@mail.tsinghua.edu.cn}).}}

\ifpdf
\hypersetup{
	pdftitle={On the convergence conditions of Laplace importance sampling with randomized quasi-Monte Carlo},
	pdfauthor={Zhan Zheng, Zhijian He, and Xiaoqun Wang}
}
\fi

\begin{document}
	
	\maketitle
	
	\begin{abstract}
	    Monte Carlo (MC) methods, renowned for their dimension-independent convergence rate of $O(N^{-1/2})$, are pivotal in computational problems. However, this rate might not always meet practical requirements. This study delves into amalgamating quasi-Monte Carlo (QMC) methods and importance sampling (IS) to enhance this rate. QMC methods, deterministic counterparts of MC, utilize low-discrepancy sequences and have found extensive applications in finance and statistics over the past three decades. Under specific conditions, QMC methods achieve an error bound of $O(\frac{{\log N}^d}{N})$ for d-dimensional integrals, surpassing the conventional MC rate. The study further explores randomized QMC (RQMC), which maintains the QMC convergence rate and facilitates computational efficiency analysis. Emphasis is laid on integrating randomly shifted lattice rules, a distinct RQMC quadrature, with IS—a classic variance reduction technique. The study underscores the intricacies of establishing a theoretical convergence rate for IS in QMC compared to MC, given the influence of problem dimensions and smoothness on QMC. The research also touches on the significance of IS density selection and its potential implications. The study culminates in examining the error bound of IS with a randomly shifted lattice rule, drawing inspiration from the reproducing kernel Hilbert space (RKHS). In the realm of finance and statistics, many problems boil down to computing expectations, predominantly integrals concerning a Gaussian measure. This study considers optimal drift importance sampling (ODIS) and Laplace importance sampling (LapIS) as common importance densities. Conclusively, the paper establishes that under certain conditions, the IS-randomly shifted lattice rule can achieve a near $O(N^{-1})$ error bound. 
	\end{abstract}
	
	
	
	\section{Introduction}
	 Monte Carlo (MC) techniques are prevalent in numerous computational challenges. A salient feature of MC is its classic $O(N^{-1/2})$ convergence rate for square-integrable integrands, given a sample size n. While this rate is dimension-agnostic, it might not always meet practical application demands. To address this, we integrate quasi-Monte Carlo (QMC) approaches with importance sampling (IS) to bolster the convergence speed.
    QMC strategies, deterministic counterparts of MC, leverage low-discrepancy sequences. Over the past three decades, their application has increased in scientific computation, notably within the finance and statistics domains. Under specific regularity conditions, the QMC quadrature achieves an error bound of $O(N^{-1+\epsilon})$ for any $\epsilon>0$ in the dimensional integral, outperforming the standard MC rate. In practice, randomized QMC (RQMC) is frequently employed. It retains QMC's convergence attributes with the added benefit of facilitating computational efficiency analysis, such as error estimation or confidence interval construction. Various RQMC quadratures exist, including scrambled digital nets and randomly shifted lattice rules. Comprehensive insights into QMC and RQMC can be found in referenced works. In prior research, we melded scrambled digital nets with IS, establishing that an approximate $O(N^{-3/2+\epsilon})$ mean square error is achievable given specific integrand boundary growth conditions and optimal IS density selection. This study pivots on randomly shifted lattice rules, a distinct RQMC variant.
    We focus on the synergy between the randomly shifted lattice rule and IS. Recognized as a pivotal variance reduction technique in MC literature, IS's application in security pricing was explored in earlier studies. Subsequent research over the years underscored its efficacy in addressing rare events. It became evident that IS's utility extended beyond the conventional variance reduction method. 
    Sampling from a desired distribution often differs from the original problem's focus. The efficacy of Importance Sampling (IS) is intrinsically tied to the choice of importance density, underscoring the significance of our selection. A comprehensive review of IS is available in \cite{Owen2013}.

    Incorporating IS within the Quasi-Monte Carlo (QMC) quadrature complicates the derivation of a theoretical convergence rate, especially when juxtaposed with Monte Carlo (MC). While MC remains largely indifferent to problem dimensions and smoothness (barring square-integrability), these factors profoundly impact QMC. Kuo \cite{Kuo2008a} demonstrated QMC's superiority over MC for log-likelihood integrals, and Dick et al.\cite{Dick2019} elucidated a weighted discrepancy bound for QMC-based IS, deriving explicit error bounds for adequately regular integrands. Our prior work highlighted potential boundary singularities in the unit hypercube due to IS density, an effect traceable to specific matrix eigenvalues. This insight spurred our exploration of error bounds for IS combined with randomly shifted lattice rules. However, our primary theoretical foundations are rooted in the Reproducing Kernel Hilbert Space (RKHS), a concept not applicable to scrambling.

    Many computational challenges in finance and statistics revolve around expectation calculations, often manifesting as Gaussian measure integrals. Common scenarios include Brownian motion-driven assets in security pricing or Gaussian priors in Bayesian computations. In determining importance density, two prevalent choices emerge, namely the Optimal Drift Importance Sampling (ODIS) and the Laplace Importance Sampling (LapIS). While ODIS employs a multivariate normal density with the original covariance matrix, LapIS opts for a general multivariate normal density, aligning its mean and covariance matrix with the integrand's mode and curvature. We also consider the multivariate t distribution as an IS proposal. Our findings indicate that the IS-randomly shifted lattice rule approximates an error bound under specific conditions.
    The paper's structure is as follows: Section 2 offers an overview of MC, QMC, IS, ODIS, LapIS, and RKHS. Section 3 delineates the theoretical error rate for RQMC-based IS estimators and extends the IS density to the t distribution. Section 4 presents examples illustrating the effects of ODIS and LapIS, and Section 5 concludes our study.
		  
	\section{Preliminaries}\label{sec:pre}
	
        \subsection{Monte Carlo and quasi-Monte Carlo method}
        Consider an integral on the unit hypercube 
        \begin{equation}
          I(g) = \int_{(0,1)^d} g(\bm z) d\bm z.
        \end{equation}
        For numerical computation, we take
        \begin{equation}\label{eq:plainsim}
          \hat{I}_N(g) = \frac 1N \sum_{i=1}^N g(\bm z_i) 
        \end{equation}
        to simulate the integral. In MC setting, the $\bm z_i'$s are randomly sampled points in the unit hypercube. If the variance of the integrand $g(\bm z)$ is finite, i.e.
        \begin{equation}\label{eq:MCvariance}
          \sigma^2 := \int_{(0,1)^d} (g(\bm z) -I(g))^2 d\bm z < \infty,
        \end{equation}
        then by central limit theorem, the root mean square error (RMSE) of MC is
        \begin{equation}
          	\sqrt{\mathbb{E}[(\hat{I}_N(g)-I(g))^2]}=\frac {\sigma}{\sqrt{N}}.
        \end{equation}
        In other words, MC has an RMSE rate $O(N^{-1/2})$ if the integrand is square-integrable. 
        
        To speed up the convergence rate, we use QMC quadrature rule to replace MC. QMC takes the same form of~\eqref{eq:plainsim}, whereas $\{\bm z_1,\dots,\bm z_N\}:=\mathcal{P}$ is a deterministic low-discrepancy point set in $(0,1)^d$ rather than independent and identically distributed (i.i.d.) points for MC. In this paper, we use lattice approach to constructing low-discrepancy point sets. The typical error bound of QMC integration is given by Koksma-Hlawka inequality \cite{Niederreiter1992}
        \begin{equation}\label{eq:hk}
            |\hat{I}_N(g)-I(g)| \leq D^*(\mathcal{P}) V_\mathrm{HK}(g).
        \end{equation}
        where $D^*(\mathcal{P})$ denotes the star discrepancy of the point set $\mathcal{P}$, while $V_\mathrm{HK}(\cdot)$ is the variation (in the sense of Hardy and Krause) of a function defined over the unit hypercube. Low-discrepancy sequences achieve a star discrepancy $O(N^{-1}(\log N)^d)$, therefore, as long as the variation is bounded, QMC integration has a deterministic error bound of $O(N^{-1}(\log N)^d)$, which is asymptotically superior to that of MC for a fixed dimension $d$.
        
	    \subsection{Importance sampling}
	    In this paper, we study integrals with respect to the normal density
	    \begin{equation}
	    	\label{eq:targetori}
	    	C = \int_{\mathbb{R}^d} G(\bm z)p(\bm z;\bm \mu_0,\bm \Sigma_0)d\bm z,
	    \end{equation}
	    where $p(\bm z;\bm \mu,\bm \Sigma)$ denotes the probability density function of $d$-dimensional normal distribution with the mean $\bm \mu$ and the covariance matrix $\bm \Sigma$. Performing IS is using a proposal function $q$ (under some regular conditions) to change the integral into
	    \begin{eqnarray}
	    	\label{eq:ISgenform}
	    	C &=& \int_{\mathbb{R}^d} G(\bm z) \frac{p(\bm z;\bm \mu_0,\bm \Sigma_0)}{q(\bm z)} q(\bm z) d\bm z \nonumber \\
	    	&=& \int_{\mathbb{R}^d} G(\bm z) W(\bm z) q(\bm z) d\bm z \nonumber \\
	    	&=& \int_{\mathbb{R}^d} G_{IS}(\bm z) q(\bm z) d\bm z,
	    \end{eqnarray}
	    where
	    \begin{equation}\label{eq:LRf}
	    	W(\bm z) = \frac{p(\bm z;\bm \mu_0,\bm \Sigma_0)}{q(\bm z)}
	    \end{equation}    
	    is the likelihood ratio (LR) function, $G_{IS}(\bm z)=G(\bm z) W(\bm z)$ is the IS integrand. 
	    
	    Firstly, we choose IS density among normal densities, i.e., $q(\bm z)=p(\bm z;\bm \mu,\bm \Sigma)$, yielding to
	    \begin{equation}\label{eq:LRnornorform}
	    	W(\bm z) = \sqrt{\frac{\det (\bm \Sigma)}{\det (\bm \Sigma_0)}}\exp(\frac 12 (\bm z-\bm \mu)^T\Sigma^{-1}(\bm z-\bm \mu)-\frac 12 (\bm z-\bm \mu_0)^T\Sigma_0^{-1}(\bm z-\bm \mu_0))
	    \end{equation}
	    To associate IS with QMC method, we need to perform
	    \begin{equation}
	    	\int_{\mathbb{R}^d} F(\bm z) p(\bm z;\bm 0,\bm I_d) d\bm z = \int_{(0,1)^d} F(\Phi^{-1}(\bm u)) d\bm u,
	    \end{equation} 
	    where $\Phi(\cdot)$ denotes the cumulative distribution function (CDF) of the $d$-dimensional standard normal distribution, $\Phi^{-1}(\cdot)$ is its inverse (applied componentwise). By the affine transformation property of normal distribution, we write
	    \begin{eqnarray}
	    	\label{eq:ISnorform}
	    	C &=& \int_{\mathbb{R}^d} G(\bm \mu+\bm L\bm z) W(\bm z)p(\bm z;\bm 0,\bm I_d)d\bm z\nonumber \\
	    	&=& \int_{\mathbb{R}^d} G_{IS}(\bm z) p(\bm z;\bm 0,\bm I_d) d\bm z \nonumber\\
	    	&=& \int_{(0,1)^d} G_{IS}(\Phi^{-1}(\bm u)) d\bm u,
	    \end{eqnarray}
	    where $G_{IS}(\bm z)=G(\bm \mu+\bm L \bm z)W(\bm z)$,$\bm L$ is any square root decomposition of covariance matrix $\bm \Sigma$ (satisfying $\bm \Sigma=\bm L \bm L^T$), and the LR function is now
	    \begin{equation}\label{eq:LRstd}
	    	W(\bm z) = \sqrt{\frac{\det (\bm \Sigma)}{\det (\bm \Sigma_0)}}\exp(\frac 12 \bm z^T \bm z-\frac 12 (\bm \mu-\bm \mu_0+\bm L \bm z)^T\Sigma_0^{-1}(\bm \mu-\bm \mu_0+\bm L \bm z)).
	    \end{equation}
	    It should be noted that one suffices to take $\bm \mu=\bm \mu_0$ and $\bm \Sigma=\bm \Sigma_0$ if IS is not applied.
	    
	    It is clear that a good choice of parameters $\bm \mu$ and $\bm \Sigma$ may reduce the calculation error. For MC simulation, that is equivalent to reducing the variance $\sigma^2$ given by~\eqref{eq:MCvariance}. If we suppose that $G$ is nonnegative for all $\bm z$, apparently the zero-variance IS density is
	    \begin{equation}
	    p_{opt}(\bm z) := \frac{1}{C}G(\bm z)p(\bm z;\bm \mu_0,\bm \Sigma_0).
	    \end{equation}
	    However, it is not practical since the unknown integral $C$ is involved. Nevertheless, we are inspired to select an IS density which mimics the behavior of $p_{opt}$.
	    
	    Let $H(\bm z) =\log (p_{opt}(\bm z))$. Assume that $H(\bm z)$ is differentiable and unimodal. Let $\bm \mu_{\star}$ be the mode of $H(\bm z)$, i.e.,
	    \begin{equation}
	    \bm \mu_{\star} = \arg \max_{\bm z \in \mathbb{R}^d} H(\bm z)= \arg \max_{\bm z \in \mathbb{R}^d} G(\bm z)p(\bm z;\bm \mu_0,\bm \Sigma_0).
	    \end{equation}
	    Taking a second-order Taylor approximation around the mode $\bm \mu_{\star}$ gives
	    \begin{equation}
	    H(\bm y) \approx H(\bm \mu_{\star}) - \frac{1}{2}(\bm y - \bm \mu_{\star})^T \bm \Sigma_{\star}^{-1}(\bm y - \bm \mu_{\star}),
	    \end{equation}
	    where
	    \begin{equation}
	    	\label{eq:solveoptvar}
	    	\bm \Sigma_{\star} = -\nabla^2 H(\bm \mu_{\star})^{-1} = (\bm\Sigma_0 - \nabla^2\log(G(\bm \mu_{\star})))^{-1}.
	    \end{equation}
	    Therefore, we obtain
	    \begin{equation}
	    	p_{opt}(\bm z) = \exp (H(\bm z)) \approx \exp\left(H(\bm \mu_{\star})-\frac{1}{2}(\bm z - \bm \mu_{\star})^T \bm \Sigma_{\star}^{-1}(\bm z - \bm \mu _{\star})\right)\propto p(\bm z;\bm \mu_{\star},\bm \Sigma_{\star}). 
	    \end{equation}
	    In other words, $p(\bm z;\bm \mu_{\star},\bm \Sigma_{\star})$ is a density which is, in a sense, close to the optimal one. Then, LapIS just takes $\bm \mu= \bm \mu_{\star}$ and $\bm \Sigma=\bm \Sigma_{\star}$. Meanwhile, ODIS chooses $\bm \mu= \bm \mu_{\star}$ but keeps $\bm \Sigma=\bm \Sigma_0$ unchanged.
	    We should note that LapIS is not necessarily better than ODIS (see e.g.,\cite{Zhang2021}\cite{He2023}). Generally speaking, in MC setting LapIS tends to effective if $p_{opt}(\bm z)$ is close to a normal density. For a QMC point of view, it becomes more complicated since the performance is in addition affected by other factors, for example, the boundary growth for scrambled digital nets. We try to analyze this question for randomly shifted lattice rule in the following sections.
	    
	    \subsection{Reproducing kernel Hilbert space}
	    RKHS is a suitable function space for error analysis, which can be reviewed from Aronszajn\cite{Aronszajn1950}. We refer to \cite{Kuo2006,Waterhouse2006,Nichols2014} for recent work.
	    
	    Consider integrals with respect to probability density which is multiplied by 1-dimension marginal density. That is
	    \begin{equation}\label{eq:martype}
	    	C(f) = \int_{\mathbb{R}^d} f(\bm x) \prod_{i=1}^d \phi(x_i)d\bm x.
	    \end{equation}
	    Let $\Phi$ be the cumulative probability function of $\phi$:
	    \begin{equation}
	    	\Phi(x) = \int_{-\infty}^x \phi(y) dy,
	    \end{equation}
	    $\Phi^{-1}$ is its inverse. For example, if $\phi(x)=p(x;0,1)$, $\Phi$ is then the cumulative probability function of standard normal distribution. Assume that $\mathcal{F}$ is an Hilbert space defined on functionals over $\mathbb{R}^d$. Then we have an isometric space $\mathcal{G}$ defined on functionals over $(0,1)^d$ by
	    \begin{equation}\label{eq:isometry}
	    	g=f\circ\Phi^{-1}.
	    \end{equation}
	    Suppose that $K_{\mathcal{F}}$ is a reproducing kernel, which means a function $K_{\mathcal{F}}:\mathbb{R}^d \times \mathbb{R}^d \rightarrow \mathbb{R}$ satisfies the following 3 properties:
	    \begin{itemize}
	    	\item property of functional: $K_{\mathcal{F}}(\cdot,\bm x) \in \mathcal{F}$ for all $\bm x \in \mathbb{R}^d$. \\
	    	\item property of symmetry: $K_{\mathcal{F}}(\bm x,\bm y) = K_{\mathcal{F}}(\bm y,\bm x)$ for all $\bm x,\bm y \in \mathbb{R}^d$. \\
	    	\item property of reproducing:$<f,K_{\mathcal{F}}(\cdot,\bm x)>=f(\bm x)$ for all $\bm x \in \mathbb{R}^d$ and $f \in \mathcal{F}$.	    	
	    \end{itemize} 
	    Furthermore, we need to assume that
	    \begin{equation}\label{eq:conditionofRKHS}
	    	K_{\mathcal{F}}(\bm x,\bm x) \prod_{i=1}^d \phi(x_i)d\bm x < \infty.
     	\end{equation}
     	This is due to
	    \begin{equation}\label{eq:RKHSform}
	    	C(f) = \int_{\mathbb{R}^d} <f,K_{\mathcal{F}}(\cdot,\bm x)> \prod_{i=1}^d \phi(x_i)d\bm x = <f,h>,
	    \end{equation}
	    where
	    \begin{equation}\label{eq:representation}
	    	h(\bm x)=\int_{\mathbb{R}^d} K_{\mathcal{F}}(\bm x,\bm y) \prod_{i=1}^d \phi(y_i)d\bm y.
	    \end{equation}
	    ~\eqref{eq:conditionofRKHS} implies that~\eqref{eq:RKHSform} and ~\eqref{eq:representation} are well-defined. Then the initial error is given by
	    \begin{equation}
	    	e_0(\mathcal{F}) = \sup_{||f||_{\mathcal{F}}||\leq 1} |C(f)|= \sqrt{<h,h>},
	    \end{equation}
	    and square both sides to obtain
	    \begin{equation}
	    	|e_0(\mathcal{F})|^2 = <h,h> \int_{\mathbb{R}^d}\int_{\mathbb{R}^d} K_{\mathcal{F}}(\bm x,\bm y) \prod_{i=1}^d [\phi(x_i)\phi(y_i)] d\bm x d\bm y.
	    \end{equation}	    
	    Recall the isometry space $\mathcal{G}$ of $\mathcal{F}$ (\ref{eq:isometry}), the integration remains the same:
	    \begin{equation}\label{eq:sameint}
	        C(f) = \int_{(0,1)^d} g(\bm u) d\bm u :=I(g).
	    \end{equation}
	    In addition, $\mathcal{G}$ is also a RKHS equipped with reproducing kernel
	    \begin{equation}
	    	K_{\mathcal{G}}(\bm u,\bm v)=K_{\mathcal{F}}(\Phi^{-1}(\bm u),\Phi^{-1}(\bm v))
	    \end{equation}
	    Thus we have the same initial error
	    \begin{equation}
	        e_0(g) = \sup_{||g||_{\mathcal{G}}||\leq 1} |I(g))| = \sup_{||f||_{\mathcal{F}}||\leq 1} |C(f)| = e_0(f).
	    \end{equation}
	    Now consider a QMC rule to estimate the integral~\eqref{eq:sameint}
	    \begin{equation}
	    	\hat{I}_N(g)=\frac 1N \sum_{i=1}^N g(\bm u_i),
	    \end{equation}
	    then the worst-case error in the space $\mathcal{G}$ is defined as
	    \begin{equation}
	    	e_w(\hat{I}_N,\mathcal{G})=\sup_{||g||_{\mathcal{G}}\leq 1} |\hat{I}_N(g)-I(g)|.
	    \end{equation}
	    Therefore, we derive the relationship between the worst-case error above and the original integral
	    \begin{equation}
	    	|\hat{I}_N(f\circ\Phi^{-1})-C(f)|=|\hat{I}_N(g)-I(g)|\leq	e_w(\hat{I}_N,\mathcal{G})||g||_{\mathcal{G}}=e_w(\hat{I}_N,\mathcal{G})||f||_{\mathcal{F}}.
	    \end{equation}
	    Furthermore, randomly shifted lattice rule ia applied in this paper. Thus the points $\bm u_i,i=1,2,...,N$ of the shifted rank-1 lattice rule are 
	    \begin{equation}
	    	u_i=\{\frac{iz}{N}+\zeta\},i=1,2,...,N.
	    \end{equation}
	    where $\{\cdot\}$ takes the fractional part (component-wise to a vector), $\bm \zeta \in (0,1)^d$ is sampled from i.i.d. vectors, $\bm z \in \mathcal{Z}_N^d$ denotes the generating vector, the notation
	    \begin{equation}
	    	\mathcal{Z}_N = \{z\in \mathbb{Z}|1\leq z\leq N-1, (z,N)=1\},
	    \end{equation}
	    i.e. the set of positive integers no more than $N-1$ which are relatively prime to $N$. Finally, we express the shift-averaged worst-case error, which only depends on the generating vector (see \cite{Sloan2002})
	    \begin{eqnarray}
	    	|e_N^{sa}(\bm z)|^2 &=& \int_{(0,1)^d} |e_w(\hat{I}_N,\mathcal{G})|^2 d\bm\zeta \nonumber \\
	    	&=& -\int_{(0,1)^d}\int_{(0,1)^d} K_{\mathcal{G}}(\bm u,\bm v)d\bm ud\bm v+\frac 1N \sum_{i=1}^N K_{\mathcal{G}}^{sik}(\{\frac{i\bm z}{N}\}),
	    \end{eqnarray}
	    where the last term $K_{\mathcal{G}}^{sik}$ denotes the shift-invariant kernel
	    \begin{equation}
	    	K_{\mathcal{G}}^{sik}(\bm u,\bm v) = \int_{(0,1)^d} K_{\mathcal{G}}(\{\bm u+\bm \zeta\},\{\bm v+\bm \zeta\})d\bm \zeta. 
	    \end{equation}
	    Note that it depends on the difference of $\bm u$ and $\bm v$, for simplicity we write
	    \begin{equation}\label{eq:sik}
	    	K_{\mathcal{G}}^{sik}(\{\bm u-\bm v\}):=K_{\mathcal{G}}^{sik}(\{\bm u-\bm v\},\bm 0)=K_{\mathcal{G}}^{sik}(\bm u,\bm v) 
	    \end{equation}
	    Shift-averaged worst-case error is the main target for error analysis of randomly shifted lattice rule. Here we introduce Fourier analysis to illustrate how to apply the theory above.
	    
	    For linear functional space $\mathcal{F}$ over $\mathbb{R}^d$, given weight function $\psi$ and weight coefficients $\gamma_k>0$, $k=1,2,...,d$, we define the inner product
	    \begin{equation}\label{eq:multiinnerprod}
	    	<f,g>_{\mathcal{F}}=\sum_{u \subset 1:d} \Big(\prod_{k \in u} \frac{1}{\gamma_k} \int_{\mathbb{R}^{|u|}} \frac{\partial^{|u|}f}{\partial \bm x_u}(\bm x_u,\bm 0) \frac{\partial^{|u|}g}{\partial \bm x_u}(\bm x_u,\bm 0) \prod_{k \in u} \psi^2(x_k) d\bm x_u \Big),
	    \end{equation}
	    where $(\bm x_u,\bm 0)$ denotes a $d$-dimensional vector whose $k$'th component takes $x_k$ if $k \in u$, otherwise takes $0$. Now the reproducing kernel of each component is given by
	    \begin{equation}\label{eq:conofRKHS}
	    	\int_{R} K_{\mathcal{F},j}(y,y) \phi(y) dy < \infty,
	    \end{equation}
	    where $K_{\mathcal{F},j}$ depends on $\gamma_j$
	    \begin{equation}\label{eq:rkj}
	    	K_{\mathcal{F},j}(x,y)=1+\gamma_j\eta(x,y),
	    \end{equation}
	    \begin{equation}\label{eq:etaexpression}
	    	\eta(x,y)= 
	    	\begin{cases}
	    		\int_0^{\max(x,y)} \frac{1}{\psi^2(t)}dt,	& x,y>0,\\
	    		\int_{\max(x,y)}^0 \frac{1}{\psi^2(t)}dt,	& x,y<0,\\
	    		0,	& \text{otherwise}.
	    	\end{cases}
	    \end{equation}
	    Then the shift-invariant kernel~\ref{eq:sik} becomes (see \cite{Kuo2006}) 
	    \begin{equation}\label{eq:sifourierform}
	    	K_{\mathcal{G}}^{sa}(\bm u,\bm v) = \sum_{u \in 1:d} \gamma_u \prod_{i \in u}  \theta(\{u_i-v_i\})), \quad \bm u,\bm v \in (0,1)^d,
	    \end{equation}
	    where
	    \begin{equation}\label{eq:fourierform}
	    	\theta(\bm x)= \int_{\Phi^{-1}(x)}^0 \frac{\Phi(t)-x}{\psi^2(t)}dt +\int_{\Omega^{-1}(1-x)}^0 \frac{\Phi(t)-1+x}{\psi^2(t)}dt,\quad x \in (0,1).
	    \end{equation}
	    Expand~\eqref{eq:fourierform} by Fourier series, we have
	    \begin{equation}\label{eq:fourierexpansion}
	    	\theta(\bm x) = \sum_{h \in \mathbb{Z}} \hat{\theta}(h)\exp(2\pi ihx),
	    \end{equation}
	    where the Fourier coefficients
	    \begin{equation}\label{eq:invfourierexpansion}
	    	\hat{\theta}(h) =\int_0^1 \theta(x)\exp(-2\pi ihx) dx.
	    \end{equation}
	    Set 
	    \begin{equation}\label{eq:C1expansion}
	    	C_1:=\int_{-\infty}^0 \frac{\Phi(t)}{\psi^2(t)}dt + \int_0^{\infty} \frac{1-\Phi(t)}{\psi^2(t)}dt.
	    \end{equation}
	    Clearly
	    \begin{equation}
	    	C_1=\theta(0)=\sum_{h \in \mathbb{Z}} \hat{\theta}(h).
	    \end{equation}
	    To avoid tedious statements, we close this section by an associated theorem.
	    \begin{theorem}\label{thm:conofRKHSequ}
	    	The well-defined condition for RKHS \eqref{eq:conofRKHS} is equivalent to
	    	\begin{equation}\label{eq:conofRKHSequ}
	    		C_1 < \infty
	    	\end{equation}	
	    \end{theorem}
	    The proof is plain by applying Fubini theorem.
	    
      \section{Main results}\label{sec:main}      
      \subsection{Estimators of ODIS and LapIS}	    
	  Again, consider the problem of estimating an integral
	  \begin{equation}
	  	\label{eq:target}
	  	C = \int_{\mathbb{R}^d} G(\bm z)p(\bm z;\bm \mu,\bm \Sigma)d\bm z.
	  \end{equation}
      Firstly we assume that $G(\bm z)>0$ for all $\bm z \in \mathbb{R}^d$. Let $g(\bm z) = \log G(\bm z)$. We rewrite~\eqref{eq:target} as
      \begin{equation}
      	\label{eq:exptype}
      	C = \int_{\mathbb{R}^d} \frac{\exp (H(\bm z))}{\sqrt{(2\pi)^d\det{(\bm \Sigma)}}}d\bm z,
      \end{equation}
      where $H(\bm z)=g(\bm z) - \frac 12 \bm z^T \bm \Sigma^{-1} \bm z$.
      
      Assume $H$ is a unimodal function. We take $\bm z_{\star} = \arg \max_{\bm z \in \mathbb{R}^d} H(\bm z)$, which solves
      \begin{equation}\label{eq:LapISdrift}
      	\nabla H(\bm z) = \nabla g(\bm z) - \bm \Sigma^{-1} \bm z =0.
      \end{equation} 
      Here we deduce the estimators of ODIS and LapIS, respectively. 
      ODIS takes importance density $p(\bm z;\bm z_{\star},\bm \Sigma)$. Therefore, the integral~\eqref{eq:target} becomes
      \begin{equation}
      	\label{eq:ODIStype}
      	\int_{\mathbb{R}^d} \frac{\exp (H(\bm z))}{\sqrt{(2\pi)^d\det{(\bm \Sigma)}}}d\bm z = (2\pi)^{-d/2} \int_{\mathbb{R}^d} \exp (H(\bm L\bm x+ \bm z_{\star})) d\bm x,
      \end{equation}
      where $\bm L \bm L^T =\bm \Sigma$.
      
      Let $\rho(\cdot)$ denote the PDF of 1-dimensional standard normal distribution. Take
      \begin{align}\label{eq:fODISform}
      	f_{O}(\bm x) &= (2\pi)^{-d/2} \exp (H(\bm L \bm x+ \bm z_{\star})) \prod_{j=1}^d \frac{1}{\rho(x_j)} \nonumber \\
      	&= \exp(g(\bm L \bm x+ \bm z_{\star})-\frac 12 (\bm L \bm x+ \bm z_{\star})^T \bm \Sigma^{-1} (\bm L\bm x+ \bm z_{\star}) +\frac 12 \bm x^T\bm x) \nonumber \\
      	&= \exp(g(\bm L \bm x+ \bm z_{\star})- \bm z_{\star}^T \bm L^{-T} \bm x - \frac 12 \bm z_{\star}^T \bm \Sigma^{-1} \bm z_{\star}) \nonumber \\
      \end{align}
      Thus we rewrite the integral as
      \begin{equation}\label{eq:ODISform}
      	C = \int_{\mathbb{R}^d} f_O(\bm x) \prod_{j=1}^d \rho(x_j) d \bm x = \int_{(0,1)^d}f_O(\bm \Phi^{-1}(\bm u)) d\bm u, 
      \end{equation}
      where $\bm \Phi^{-1}(\cdot)$ is the inverse (applied componentwise) of CDF of d-dimensional standard normal distribution. Let
      \begin{equation}\label{eq:estimatorofODIS}
      	\hat{I}_N(G_{ODIS}) = \frac 1N \sum_{i=1}^N  f_O(\bm \Phi^{-1}(\bm u_i)),
      \end{equation}
      where $\bm u_i^{'}s$ are from the unit hypercube. For MC they are i.i.d., and here for QMC associated with lattice rule, they can be constructed by rank-1 randomly shifted lattice rule (see Algorithm 6 of \cite{Nichols2014}, CBC algorithm). 
       
      Next consider LapIS. Let $\bm \Sigma_{\star} = (-\nabla^2 H(z_{\star}))^{-1}$, where
      $\nabla^2 H(\bm z)= \nabla^2 g(\bm z) - \bm \Sigma^{-1}$, or equivalently
      \begin{equation}\label{eq:LapIScovariance}
      	\bm \Sigma^{-1} = \bm \Sigma_{\star}^{-1} + \nabla^2 g(\bm z)
      \end{equation}
      Suppose $\bm L_{\star}$ is a factorization matrix of $\bm \Sigma_{\star}$ satisfies $\bm L_{\star} \bm L_{\star}^T=\bm \Sigma_{\star}$. Therefore, the integral~\eqref{eq:target}  becomes
      \begin{equation}
      	\label{eq:LapIStype}
      	\int_{\mathbb{R}^d} \frac{\exp (H(\bm z))}{\sqrt{(2\pi)^d\det{(\bm \Sigma)}}}d\bm z = (2\pi)^{-d/2} \frac{\det (\bm L_{\star})}{\sqrt{\det{(\bm \Sigma)}}} \int_{\mathbb{R}^d} \exp (H(\bm L_{\star} \bm x+ \bm z_{\star})) d\bm x.
      \end{equation}
      Let $\rho(\cdot)$ still denote the PDF of 1-dimensional standard normal distribution. Take
      \begin{align}\label{eq:fLapISform}
      	f_{L}(\bm x) &= (2\pi)^{-d/2} \frac{\det (\bm L_{\star})}{\sqrt{\det{(\bm \Sigma)}}}\exp (H(\bm L_{\star} \bm x+ \bm z_{\star})) \prod_{j=1}^d \frac{1}{\rho(x_j)} \nonumber \\
      	&= \frac{\det (\bm L_{\star})}{\sqrt{\det{(\bm \Sigma)}}} \exp(g(\bm L_{\star} \bm x+ \bm z_{\star})-\frac 12 (\bm L_{\star} \bm x+ \bm z_{\star})^T \bm \Sigma^{-1} (\bm L_{\star} \bm x+ \bm z_{\star}) +\frac 12 \bm x^T\bm x) \nonumber \\
      	&= \frac{\det (\bm L_{\star})}{\sqrt{\det{(\bm \Sigma)}}} \exp(g(\bm L_{\star} \bm x+ \bm z_{\star})-\frac 12 (\bm L_{\star} \bm x+ \bm z_{\star})^T (\bm \Sigma_{\star}^{-1}+\nabla^2 g(\bm z_{\star})) (\bm L_{\star} \bm x+ \bm z_{\star}) +\frac 12 \bm x^T\bm x) \nonumber \\
      	&= \frac{\det (\bm L_{\star})}{\sqrt{\det{(\bm \Sigma)}}} \exp(g(\bm L_{\star} \bm x+ \bm z_{\star})-\frac 12 (\bm L_{\star} \bm x+ \bm z_{\star})^T \nabla^2 g(\bm z_{\star}) (\bm L_{\star} \bm x+ \bm z_{\star}) - \bm z_{\star}^T\bm \Sigma_{\star}^{-1}\bm L_{\star}\bm x -\frac 12 \bm z_{\star}^T\bm \Sigma_{\star}^{-1} \bm z_{\star}). \nonumber \\
      \end{align}
      The third equation follows from \eqref{eq:LapIScovariance}. Finally, we rewrite the integral as
      \begin{equation}\label{eq:LapISform}
      	C = \int_{\mathbb{R}^d} f_L(\bm x) \prod_{j=1}^d \rho(x_j) d \bm x = \int_{(0,1)^d} f_L(\bm \Phi^{-1}(\bm u)) d\bm u.
      \end{equation}
      Let
      \begin{equation}\label{eq:estimatorofLapIS}
      	\hat{I}_N(G_{LapIS}) = \frac 1N \sum_{i=1}^N  f_L(\bm \Phi^{-1}(\bm u_i)),
      \end{equation}
      where $\bm u_i^{'}s$ are the same as the previous ODIS part. 
      
      \subsection{Error bound for importance sampling}
      Nichols\cite{Nichols2014} connected the Fourier analysis and shift-averaged worst-case error, claimed that
      \begin{lem}(Kuo,2010)\label{lem:Kuo2010} \quad
      	If there exist constants $C_2>0$ and $r_2>1/2$, such that for any $h \in \mathbb{Z}$ and $h \neq 0$ holds
      	\begin{equation}\label{eq:Kuoupperboundcon}
      		\hat{\theta}(h) \leq \frac{C_2}{|h|^{2r_2}},
      	\end{equation}
      	then the generating vector $\bm z$ constructed by component-by-component algorithm satisfies
      	\begin{equation}\label{eq:Kuoupperboundsol}
      		e_{N}^{sa}(\bm z)=O(N^{-r_2+\delta}), \quad \delta>0.
      	\end{equation}
      \end{lem}
      We fix CBC algorithm and POD weights (see e.g. \cite{Kuo2011,Kuo2012}), and leave the research of weight coefficients for future work. For nonzero $h$, after some algebra (which is deferred to the Appendix) we can obtain
      \begin{equation}\label{eq:invfourierafter}
      	\hat{\theta}(h) = \frac{1}{\pi^2h^2} \int_0^1 \frac{\sin^2(\pi ht)}{\psi^2(\Phi^{-1}(t))\phi(\Phi^{-1}(t))} dt.
      \end{equation}
      For classic ODIS or LapIS, we take $\phi(x)=p(x;0,1)$ and weight function $\psi(x)=\exp(-\frac{x^2}{2\alpha^2})$. Since both of them are symmetric about $x=0$, we can rewrite~\eqref{eq:invfourierafter} as
      \begin{equation}\label{eq:invfouriersym}
      	\hat{\theta}(h) = \frac{2}{\pi^2h^2} \int_0^{1/2} \frac{\sin^2(\pi ht)}{\psi^2(\Phi^{-1}(t))\phi(\Phi^{-1}(t))} dt.
      \end{equation}
      Substitute $\phi$ and $\psi$,
      \begin{equation}\label{eq:invfouriernornor}
      	\hat{\theta}(h) = \frac{2\sqrt{2\pi}}{\pi^2h^2} \int_0^{1/2} \exp\big((1+\frac{2}{\alpha^2})\frac{(\Phi^{-1}(u))^2}{2}\big) \sin^2(\pi ht) dt.
      \end{equation}
      Note the fact that for any $0<u\leq 1/2$, $\exp(\frac{[\Phi^{-1}(u)]^2}{2}) \leq \frac 1u$. Therefore, we scale the right hand side of~\eqref{eq:invfouriernornor}
      \begin{equation}\label{eq:invfouriernornormid}
      	\hat{\theta}(h) \leq \frac{2\sqrt{2\pi}}{\pi^2h^2} \int_0^{1/2} u^{-1-2/\alpha^2}  \sin^2(\pi ht) dt
      \end{equation}
      Moreover, for any positive integer $h$ and $0<c<1$, we have
      \begin{equation}\label{eq:ubofsin2}
      	\int_0^{1/2} u^{-1-c}\sin^2(\pi ht) dt \leq \frac{(2\pi h)^c}{c(2-c)}.
      \end{equation}
      Therefore, for $\alpha^2>2$ (i.e. $0<\frac{2}{\alpha^2}<1$) and any positive integer $h$:
      \begin{equation}
      	\hat{\theta}(h) \leq \frac{\sqrt{2}\pi^{\frac{2}{\alpha^2}-\frac{3}{2}}}{(\alpha^2-1)} h^{-2(1-1/\alpha^2)}.
      \end{equation}
      Finally, by symmetry of Fourier coefficient, we claim that for nonzero integer $h$, the condition ~\eqref{eq:Kuoupperboundcon} of lemma~\ref{lem:Kuo2010} holds, where
      \begin{equation}
      	C_2 = \frac{\sqrt{2}\pi^{\frac{2}{\alpha^2}-\frac{3}{2}}}{(\alpha^2-1)}, \quad r_2 = 1-1/\alpha^2, \quad \alpha^2>2.
      \end{equation}
      Thus, due to \eqref{eq:Kuoupperboundsol}, the shift-averaged worst-case error associated with normal proposal and weight function is given by
      \begin{equation}
      	e_{N}^{sa}(\bm z)=O(N^{-(1-1/\alpha^2)+\delta}).
      \end{equation}
      The hidden condition remained is that the IS integrand locates in the RKHS. We list up some conditions for the for log original integrand $g$.
      \begin{assum}[Upper bound]\label{assum:originalboundary}
      	\begin{equation*}
      		g(\bm z) = O(||\bm z||^{\beta}), \beta < 2,
      	\end{equation*}   	
      \end{assum}
      \begin{assum}[Minimax eigenvalue]\label{assum:eigenvalue}
      	\begin{equation*}
      		-\min_{i \in 1:d} h_i \max_{j \in 1:d} \lambda_j< \frac{1}{\alpha}.
      	\end{equation*}
      	where $h_i$ denote eigenvalues of $\nabla^2 g(\bm z_{\star})$, and $\lambda_k$ denote eigenvalues of $\Sigma_{\star}, k=1,2,...,d$.
      \end{assum}
            \begin{assum}[Semi-positive definite]\label{assum:convex}
      	$\nabla^2 g(\bm z_{\star})$ is semi-definite.
      \end{assum}
      Now we have prepared to state main results for ODIS and LapIS.
      \begin{theorem}\label{thm:alternative}
      	Consider the integral \eqref{eq:target}. Apply RQMC associated with rank-1 randomly shifted lattice rule constructed by CBC algorithm and equipped by weight function $\psi(\bm x) \equiv \exp(-\frac{x^2}{2\alpha^2})$, where $\alpha^2>2$ makes the corresponding reproducing kernel Hilbert space well-defined. We claim
      	\begin{equation}\label{eq:errorbound}
      		\sqrt{\mathbb{E}|(\hat{I}_N(G_{IS})-C|^2}= O(N^{-1+1/\alpha^2+\delta}), \quad \delta>0.
      	\end{equation} 
      	For ODIS, a sufficient condition for~\eqref{eq:errorbound} is Assumption~\ref{assum:originalboundary}. For LapIS, sufficient conditions for ~\eqref{eq:errorbound} are Assumptuion~\ref{assum:originalboundary}+Assumptuion~\ref{assum:eigenvalue}.
      \end{theorem}
      \begin{proof}
      As long as $\beta < 2$, Assumptuion~\ref{assum:originalboundary} makes $g(\bm L_{\star} \bm x+ \bm z_{\star})$ dominated by $\frac {(\bm x^T \bm x)}{2\alpha}$, implying $\exp(g(\bm L_{\star} \bm x+ \bm z_{\star})-\frac {(\bm x^T \bm x)}{2\alpha})$ vanishes as $\bm x$ goes to infinity. For ODIS,~\eqref{eq:fODISform} ensures that Assumption~\ref{assum:originalboundary} is sufficient to make $f_O$ locate in RKHS. However, for LapIS, there exist some terms more for~\eqref{eq:fLapISform}. More specifically, we pick up the dominating term $\exp[D(\bm x)]$, where
      \begin{equation}\label{eq:domination}
	    D(\bm x) = g(\bm L_{\star} \bm x+ \bm z_{\star}) - \frac 12 (\bm L_{\star} \bm x)^T \nabla^2 g(\bm z_{\star})(\bm L_{\star} \bm x)
      \end{equation}
      Assumptuion~\ref{assum:originalboundary} controls the first term. If Assumptuion~\ref{assum:eigenvalue} is added, let $\nabla^2 g(\bm z_{\star}) = \bm Q^T \bm T \bm Q$ denote the singular value decomposition. Then $\bm Q$ is an orthogonal matrix and $\bm T$ is a diagonal matrix with eigenvalues being diagonal entries. Furthermore,
      \begin{equation}
      	- \frac 12 (\bm L_{\star} \bm x)^T \nabla^2 g(\bm z_{\star})(\bm L_{\star} \bm x) = -\frac 12 (\bm Q \bm L_{\star} \bm x)^T \bm T (\bm Q \bm L_{\star} \bm x),
      \end{equation}
      which is dominated by 
      \begin{equation}
      	-\frac{h}{2} ||\bm Q \bm L_{\star} \bm x||^2 \leq -\frac{h}{2} ||\bm Q||^2_2||\bm L_{\star}||^2_2||x||^2 \leq -\frac{h}{2} \max_{j \in 1:d} \lambda_j ||x||^2.
      \end{equation}
      The second inequality holds since the spectral norm of orthogonal matrix is 1 and the spectral norm of $\bm L_{\star}$ is the nonnegative square root of maximal eigenvalue of $\bm L_{\star}^T \bm L_{\star}$, which is the same as $\bm L_{\star} \bm L_{\star}^T = \bm \Sigma_{\star}$.
      Note that Assumptuion~\ref{assum:eigenvalue} is equivalent to $\frac{1}{2\alpha} > -\frac{h}{2} \max_{j \in 1:d} \lambda_j$, then the integrand still belongs to RKHS. Our claim holds.           
      \end{proof}	
      If Assumption~\ref{assum:eigenvalue} is reckoned to be sophisticated, one can use Assumption~\ref{assum:convex} to replace it since the latter is sufficient for the former. 

      \begin{rem}
      	We consider $\beta$ in Assumptuion~\ref{assum:originalboundary}. If $\beta =2$, then 
      	\begin{eqnarray}
      		g(\bm L_{\star} \bm x+ \bm z_{\star}) &=& O(||\bm L_{\star} \bm x||^2) \nonumber \\
      		&\leq& O(||\bm L_{\star}||^2||\bm x||^2) \nonumber \\
      		&=& O((\max_{j \in 1:d} \lambda_j)||\bm x||^2) \nonumber \\
      	\end{eqnarray}
        Since the RKHS contains $\exp(\gamma ||\bm x||^2)$ when $\gamma < \frac{1}{2\alpha}$, we have $\max_{j \in 1:d} \lambda_j < \frac{1}{2\alpha}$. Therefore, a more accurate statement of Assumptuion~\ref{assum:originalboundary} is that
        \begin{equation}
            g(\bm z) = O(||\bm z||^{\beta}), \beta < 2
        \end{equation}
        or 
        \begin{equation}\label{assum:originalboundaryplus}
        	g(\bm z) = O(||\bm z||^2), \quad \max_{j \in 1:d} \lambda_j < \frac{1}{2\alpha}
        \end{equation}
      \end{rem}
      \begin{rem}
        Let us consider weight functions. For LapIS, in fact we do not have many choices. For example, we cannot take $\psi(\bm x) \equiv \exp(-\frac{|x|}{\alpha})$ as the weight function. Because LapIS brings $\exp(- \frac 12 (\bm L_{\star} \bm x)^T \nabla^2 g(\bm z_{\star})(\bm L_{\star} \bm x))$, which may not belong to the corresponding RKHS, leading the conclusion failed. However, it is not the case for ODIS since the optimal drift does not produce singularities. Hence, we have more choices of weight function for ODIS than LapIS. It turns out that the conditions of convergence for LapIS are more strict than ODIS.
      \end{rem}
	  
	  \subsection{Multivariate $t$ distribution as the proposal}
	  We take multivariate $t$ distribution as the proposal instead of normal distribution in this section. The multivariate $t$ distribution with center $\bm\mu\in\mathbb{R}^{d\times1}$, scale (positive definite) matrix $\bm\Sigma\in\mathbb{R}^{d\times d}$ and $\nu>0$ degrees of freedom, denoted by $t_\nu(\bm\mu,\bm\Sigma)$ has a representation
	  $$\bm t = \bm \mu + \bm L\bm x,$$
	  where  $\bm L\bm L^T = \bm \Sigma$, and $\bm x $ follows a multivariate t-distribution with independent components.
	  \begin{equation}
	  	q(\bm x;\bm 0,\bm I_d,\nu) = c_{\nu}\prod_{j=1}^d(1+\frac{x_j^2}{\nu})^{-\frac{\nu + 1}{2}}
	  \end{equation}
	  where $c_{\nu}$ denotes for normalization constant. 
	  By a change of variables, we have
	  \begin{eqnarray}\label{eq:rstrational}
	  	C &=& \int_{\mathbb{R}^d} G(\bm t)p(\bm t;\bm \mu_0,\bm \Sigma_0) d\bm t \nonumber \\
	  	&=& \int_{\mathbb{R}^d} G(\bm t)\frac{p(\bm t;\bm \mu_0,\bm \Sigma_0)}{q(\bm t;\bm \mu,\bm \Sigma, \nu)}q(\bm{\mu+Lx}; \bm \mu,\bm \Sigma,\nu) d\bm t \nonumber \\
	  	&=& \int_{\mathbb{R}^d} G(\bm{\mu+Lx})\frac{p(\bm{\mu+Lx};\bm \mu_0,\bm \Sigma_0)}{q(\bm{\mu+Lx}; \bm \mu,\bm \Sigma,\nu)}q(\bm x;\bm 0,\bm I_d,\nu) d\bm x \nonumber \\
	  	&=& \int_{\mathbb{R}^d} G(\bm{\mu+Lx})W(\bm x)q(\bm x;\bm 0,\bm I_d,\nu) d\bm x,
	  \end{eqnarray}
	  where $\bm L^T \bm L=\bm \Sigma$, the likelihood ratio function is given by
	  \begin{eqnarray}\label{eq:LRoftrational}
	  	W(\bm x) &=& \frac{p(\bm{\mu+Lx};\bm \mu_0,\bm \Sigma_0)}{q(\bm{\mu+Lx}; \bm \mu,\bm \Sigma,\nu)} \nonumber \\
	  	&=& \frac{(\frac{\nu}{2})^{d/2}\Gamma(\frac{\nu}{2})}{\Gamma(\frac{\nu+d}{2})(\det{\Sigma_0})^{1/2}}\prod_{j=1}^d(1+\frac{x_j^2}{\nu})^{-\frac{\nu + 1}{2}}\exp\{-(\bm \mu+\bm L \bm x -\bm \mu_0)^T\bm \Sigma_0^{-1}(\bm \mu+\bm L \bm x -\bm \mu_0)/2\} \nonumber \\
	  	&\propto& \prod_{j=1}^d(1+\frac{x_j^2}{\nu})^{-\frac{\nu + 1}{2}}\exp\{-\frac 12 \bm x^T \bm L^T \bm \Sigma_0^{-1} \bm L \bm x-\bm L^T \bm \Sigma_0^{-1}(\bm \mu -\bm \mu_0) \bm x\}, \nonumber
	  \end{eqnarray}
	  $\Gamma(\cdot)$ is the well-known Gamma function. To generate $t$-variables, we set
	  \begin{equation}\label{eq:repofz}
	  	\bm z= T^{-1}(u).
	  \end{equation}
        where $T$ denotes for the CDF of multivariate t-distribution with independent components. 
	  One can simulate variables from $(0,1)^{d}$ via ~\eqref{eq:repofz}. Hence the estimator for multivariate $t$ distribution IS is
	  \begin{equation}\label{eq:rstproposal}
	  	\hat{I}_N(G_t)=\frac 1N \sum_{i=1}^N f_t(T^{-1}(u_i)),
	  \end{equation}
	  where $f_t(\bm x)=G(\bm \mu +\bm L \bm x)W(\bm x)$, $\bm u_1,...,\bm u_N$ are random points for MC or randomly shifted lattice rule for RQMC on $(0,1)^{d}$.
	  When we analyze the error bound for randomly shifted lattice rule, the choice of weight function is crucial. Here $t$ distribution is taken as the proposal function, we cannnot still take the normal density as the weight function since the corresponding RKHS does not contain the density of $t$ distribution. An alternative selection is the rational function
	  \begin{equation}
	  	\rho_{\lambda}(x)=\frac{\lambda-1}{2}\frac{1}{(1+|x|)^{\lambda}}, \quad \lambda>1.
	  \end{equation} 
	  The tail is heavier such that $t$ distribution with degrees of freedom $\nu>2\lambda+1$ locates in the RKHS. We have the following result.
	  \begin{theorem}\label{thm:tproposal}
	  	Consider the integral \eqref{eq:target}. Apply RQMC associated with rank-1 randomly shifted lattice rule constructed by CBC algorithm and equipped by weight function $\psi(\bm x)=\prod_{i=1}^{d} \rho_{\lambda}(x_i)$ , where $\nu>2\lambda+1$ makes the corresponding reproducing kernel Hilbert space well-defined. Then a sufficient condition for
	  	\begin{equation}\label{eq:errorboundoft}
	  		\sqrt{\mathbb{E}|(\hat{I}_N(G_t)-C|^2}= O(N^{-1+\frac{2\lambda+1}{2\nu}+\delta}), \quad \delta>0.
	  	\end{equation} 
	  	is Assumption~\ref{assum:originalboundary}.
	  \end{theorem}
	  \begin{proof}
	    We start from the Fourier coefficient~\eqref{eq:invfouriersym}. The difficulty is that there is no analytic expression of CDF of $t$ distribution, therefore we try to give an upper bound of the integrand. Note that
	    \begin{equation}
	    	\min(1,\nu) \leq \frac{(1+|x|)^2}{1+x^2/\nu} \leq 1+\nu,
	    \end{equation}
	    Denote the normalizing constant $c_{\nu}=\frac{\Gamma(\frac{\nu+1}{2})}{\sqrt{\nu\pi}\Gamma(\frac{\nu}{2})}$, we thus have
	    \begin{equation}\label{eq:studentsppdf}
	    	L_{\nu} \leq \frac{\phi(x)}{\rho_{\nu}(x)} \leq U_{\nu},
	    \end{equation}
	    where the lower constant
	    \begin{equation}
	    	L_{\nu}=\frac{2c_{\nu}}{\nu}\big(\min(1,\nu)\big)^{\frac{\nu+1}{2}},
	    \end{equation}
	    and the upper constant
	    \begin{equation}
	    	U_{\nu}=\frac{2c_{\nu}}{\nu}(1+\nu)^{\frac{\nu+1}{2}}
	    \end{equation}
	    are irrelevant to $x$, and obviously $L_{\nu}<1<U_{\nu}$. Take the upper limit integral with respect to~\eqref{eq:studentsppdf}
	    \begin{equation}\label{eq:studentspcdf}
	    	L_{\nu} \leq \frac{\Phi(x)}{P_{\nu}(x)} \leq U_{\nu},
	    \end{equation}
	    where
	    \begin{equation}
	    	P_{\nu}(x) = \int_{-\infty}^{x} \rho_{\nu}(t) dt =
	    	\begin{cases}
	    		\frac{1}{2(1-x)^2}, & x<0 \\
	    		1-\frac{1}{2(1+x)^2}, & x\geq 0\\
	    	\end{cases}
	    \end{equation}
	    Let $x=\Phi^{-1}(u),\quad 0<u<1$,
	    \begin{equation}
	    	\frac{u}{U_{\nu}} \leq P_{\nu}(\Phi^{-1}(u)) \leq \frac{u}{L_{\nu}}
	    \end{equation}
	    Take $P_{\nu}^{-1}$ on both sides and note that the right hand side may be larger than 1, we have
	    \begin{equation}\label{eq:studentspinverse}
	    	P_{\nu}^{-1}(\frac{u}{U_{\nu}}) \leq \Phi^{-1}(u) \leq P_{\nu}^{-1}(\min(\frac{u}{L_{\nu}},1))
	    \end{equation}
	    Hence, for $0<t\leq 1/2$:
	    \begin{eqnarray}\label{eq:hatthetaofstudentrational}
	    	\psi^2(\Phi^{-1}(t))\phi(\Phi^{-1}(t))
	    	&\geq&
	    	\psi^2(P_{\nu}^{-1}(\frac{t}{U_{\nu}}))L_{\nu}\rho_{\nu}(\Phi^{-1}(t)) \nonumber \\
	    	&\geq& \psi^2(P_{\nu}^{-1}(\frac{t}{U_{\nu}}))L_{\nu}\rho_{\nu}(P_{\nu}^{-1}(\frac{t}{U_{\nu}})) \nonumber \\
	    	&=&
	    	\Big(\frac{2t}{U_{\nu}}\Big)^{2\lambda/\nu}L_{\nu}\frac{\nu}{2}\Big(\frac{2t}{U_{\nu}}\Big)^{1+1/\nu} \nonumber \\
	    	&=& \frac{\nu L_{\nu}}{2} \Big(\frac{2t}{U_{\nu}}\Big)^{1+(2\lambda+1)/\nu} \nonumber \\
	    \end{eqnarray}
	    where the inequality is due to the monotonicity of $\psi$.Combine ~\eqref{eq:invfouriersym} and ~\eqref{eq:hatthetaofstudentrational}, we finally derive
	    \begin{equation}
	    	\hat{\theta}(h) \leq \frac{C_t}{h^2} \int_0^{1/2} t^{-1-\frac{2\lambda+1}{\nu}} \sin^2(\pi ht) dt. 
	    \end{equation}
	    According to~\eqref{eq:ubofsin2}, since we have demonstrated that for any nonzero integer $h$, $\hat{\theta}(h) \leq C(\lambda,\nu)h^{-2r_{\star}}$, where the constant $C(\lambda,\nu)$ only depends on $\lambda$ and $\nu$, $r_{\star}=1-\frac{2\lambda+1}{2\nu}$. By lemma \ref{lem:Kuo2010}, the error bound holds.
	    We conclude our proof by claiming that $f_t$ locate in the RKHS. This fact is trivial since the likelihood ratio contains a negative quadratic form in the exponent term, whereas the ~\eqref{assum:originalboundary} ensures the boundedness of $f_t$.	    
	  \end{proof}
	  Since $\nu>2\lambda+1$, the convergence rate is no worse than MC. Furthermore, we find out that as the degree of freedom goes to infinity, the error bound gets close to $O(N^{-1+\delta})$, which corresponds with the normal importance density, i.e. the limit case of $t$ distribution.
	  
	  \section{Examples}\label{sec:exam}
	  \subsection{Generalized linear mixed model}
	  Consider a class of highly structured models in statistics, which is known as generalized linear mixed model (GLMM, see \cite{Kuo2008a}). To make it short, we express the model as
	  \begin{equation}\label{eq:GLMM}
	  	L(\beta,\kappa,\sigma) = \int_{\mathbb{R}^d} \prod_{j=1}^d \frac{\exp(y_j(\omega_j+\beta)-\exp(\omega_j+\beta))}{y_j}\frac{\exp(-\frac{1}{2} \bm \omega^T\bm \Sigma^{-1}\bm \omega^T)}{\sqrt{(2\pi)^d\det(\bm \Sigma)}} d\bm \omega,
	  \end{equation}
      where $\bm y = (y_1,y_2,...,y_d)$ denotes the data of nonnegative integers, $\Sigma$ denotes the covariance matrix, whose entries $\Sigma_{i,j}=\frac{\sigma^2 \kappa^{|i-j|}}{1-\kappa^2}$ for $i,j = 1,2,...,d.$ Our goal is to maximize the log-likelihood $\log L(\beta,\kappa,\sigma)$ given $\bm y$ with respect to $(\beta,\kappa,\sigma)$ under restrictions $\kappa \in (0,1), \sigma>0$.
      Note the integral ~\eqref{eq:GLMM} can be rewritten by $\int_{\mathbb{R}^d} \exp(F(\bm \omega)) d\bm \omega$ regardless of a normalizing constant, where
      \begin{equation}\label{eq:unimodalF}
      	F(\bm \omega) = \sum_{j=1}^d (y_j(\omega_j+\beta)-\exp(\omega_j+\beta)) - -\frac{1}{2} \bm \omega^T\bm \Sigma^{-1}\bm \omega^T.
      \end{equation}
  
      Obviously $F$ is a unimodal function, thus we take
      \begin{equation}
      	\bm \omega_{\star} = \arg \max_{\bm \omega \in \mathbb{R}^d} F(\bm \omega),
      \end{equation}
      which solves
      \begin{equation}\label{eq:solofdrift}
      	\nabla F(\bm \omega) = \bm y - e^{\beta}\exp(\bm \omega) - \bm \Sigma^{-1} \bm \omega = \bm 0.
      \end{equation}
  
      The Hessian is
      \begin{equation}\label{eq:Hessian}
      	\nabla^2 F(\bm \omega^{\star}) = -e^{\beta}\diag{e^{\omega_{\star 1}},e^{\omega_{\star2}},...,e^{\omega_{\star d}}}-\bm \Sigma^{-1}
      \end{equation}
      Recall the denotation in LapIS, let 
      \begin{equation}\label{eq:solofcovariance}
      	\bm \Sigma_{\star}= (-\nabla^2 F(\bm \omega_{\star}))^{-1},
      \end{equation}
      and take the decompostion of $\bm \Sigma_{\star}$, that is $\bm \Sigma_{\star}=\bm L_{\star} \bm L_{\star}^T$, then the integral $\int_{\mathbb{R}^d} \exp F(\bm \omega) d\bm \omega$ becomes
      \begin{eqnarray}\label{eq:LapISofGLMM}
      	\int_{\mathbb{R}^d} \exp(F(\bm \omega)) d\bm \omega &=& \det(\bm L_{\star}) \int_{\mathbb{R}^d} \exp(F(\bm L_{\star}\bm v)+\bm \omega_{\star}) d\bm v \nonumber \\
      	&=& \det(\bm L_{\star}) \int_{(0,1)^d} \exp(F(\bm L_{\star}\Phi^{-1}(\bm u)+\bm w_{\star}))\prod_{j=1}^d \frac{1}{\rho(\Phi^{-1}(u_j))} d\bm u
      \end{eqnarray}
      where $\rho$ and $\Phi$ denote the PDF and CDF of standard normal distribution, respectively.
      
      According to \cite{Kuo2010} and previous analysis, we take the weight function $\psi(\bm x)= \exp (-\frac{\bm x^2}{2\alpha^2})$ to establish the RKHS with $\alpha^2 > 2$.
      \begin{thm}
      	 There exists a randomly shifted lattice rule which attains $O(N^{-1+1/\alpha^2+\delta})$ error bound for this generalized linear mixed model if the integrand belongs to the RKHS, which is equivalent to
        \begin{equation}\label{eq:sufofGLMM}
      	  \max_{j \in 1:d} \lambda_j< \frac{\exp(-\beta)}{\alpha^2 \exp(\max_{j \in 1:d} \omega_{\star j} )}
        \end{equation}
      \end{thm}
      \begin{proof}
        Let 
        \begin{eqnarray}
          g(\bm\omega) &=& \log\left(\prod_{j=1}^d \frac{\exp(y_j(\omega_j+\beta)-\exp(\omega_j+\beta))}{y_j}\right) \nonumber \\
          &=& \sum_{j=1}^d \left[y_j(\omega_j+\beta)-\exp(\omega_j+\beta)-\log y_j\right]. 	
        \end{eqnarray}
        Obviously $g$ satisfies Assumptuion~\ref{assum:originalboundary} since it is upper bounded. Take the derivatives with respect to the components, we have
        \begin{equation}
          \frac{\partial g(\bm\omega)}{\partial \omega_k}= y_k-	\exp(\omega_k+\beta),
        \end{equation}
        \begin{equation}
        	\frac{\partial^2 g(\bm\omega)}{\partial \omega_j\omega_k}=-\exp(\omega_k+\beta)\delta_{jk}.
        \end{equation}
        Note that $\nabla^2 g$ is a negative definite diagonal matrix, by Theorem~\ref{thm:alternative}, it is sufficient to verify whether~\ref{assum:eigenvalue} is valid. Clearly it is equivalent to \eqref{eq:sufofGLMM}.
      \end{proof}
      \begin{rem}
        Since $\alpha^2>2$ for RKHS well-defined, we transpose some terms in \eqref{eq:sufofGLMM} and immediately obtain a necessary condition:
        \begin{equation}
          2 < \alpha^2 < \frac{\exp(-\beta)}{\max_{j \in 1:d} \lambda_j \exp(\max_{j \in 1:d} \omega_{\star j})} 	
        \end{equation}
         or equivalently
         \begin{equation}\label{eq:necofGLMM}
         	2e^\beta\max_{j \in 1:d} \lambda_j \exp(\max_{j \in 1:d} \omega_{\star j}) <1.
         \end{equation}
         This is not always valid, especially when $d$ is large. The reason is that LapIS may make the boundary unbounded, even diverge faster than $\exp(\bm x^T\bm x/2\alpha^2)$, which is unfavorable in QMC computation. 
      \end{rem}

      \subsection{Randleman-Bartter model}
      Consider the problem of valuing a zero coupon bond, where the interest rates are assumed to follow the Randleman-Bartter model(see \cite{Hull2005}). A 3-dimensional Gaussian integral is analyzed by Catflisch\cite{Caflisch1998}. Here we reformulate the model in general. Our goal is to value a fair price of a $(d+1)-$year zero coupon boud with a face value of \$1, which can be expressed as a integral
      \begin{equation}\label{eq:RBmodel}
      	P = \int_{\mathbb{R}^d} \prod_{k=0}^d \frac{p(\bm z;\bm 0,\bm I_d)}{1+r_k}d\bm z.
      \end{equation}
      The interest rates can be represented by
      \begin{equation}
      	r_k = r_0\exp(-k\sigma^2/2+\sigma B_k),\quad k=1,2,...,d,
      \end{equation}
      where the Brownian vector $(B_1,B_2,...,B_d)^T$ is zero-mean and the covariance matrix $\bm C$ with entries $C_{i,j}=\min(i,j)$. $r_0$ is the current annually interest rate, $\sigma$ is the volatility.
      
      The first step is to generalize the Brownian motion. This is equivalent to find a matrix $\bm A$ such that $\bm A \bm A^T= \bm C$. Here we take the standard construction to make an analysis.
      In fact, there exist several constructions which are widely used in practice such as Brownian bridge and principal component analysis (see \cite{Hull2005}). We leave other generation methods for future work.
      
      Standard construction gives
      \begin{equation}\label{eq:STD}
      	\bm A := (a_{i,j})_{d\times d} = 
      	\begin{pmatrix}
      		1 & 0 & 0 & \cdots & 0 \\
      		1 & 1 & 0 & \cdots & 0 \\
      		1 & 1 & 1 & \cdots & 0 \\
      		\vdots & \vdots & \vdots & \ddots & \vdots \\
      		1 & 1 & 1 & \cdots & 1
      	\end{pmatrix}
      \end{equation}
      and we take 
      \begin{equation}\label{eq:BMgeneration}
      	(B_1,B_2,...,B_d)^T = \bm A (z_1,z_2,...,z_d)^T,\quad (z_1,z_2,...,z_d)^T \sim N(\bm 0,\bm I_d).
      \end{equation}
      Next we rewrite~\eqref{eq:RBmodel} as standard form~\eqref{eq:target}, that is 
      \begin{equation}
      	G(\bm z) = \prod_{k=0}^d (1+r_k)^{-1}, \quad \bm \mu = \bm 0, \quad \bm \Sigma = \bm I_d.
      \end{equation}
      Note that $a_{i,j}=1$ if and only if $i \geq j$, otherwise $a_{i,j}=0$. Therefore,
      \begin{equation}
      	\frac{\partial r_k}{\partial z_i} = \sigma r_k \bm 1\{k \geq i\}.
      \end{equation}
      \begin{equation}
      	\frac{\partial^2 r_k}{\partial z_iz_j} = \sigma^2 r_k \bm 1\{k \geq \max(i,j)\}.
      \end{equation}
      Let $g(\bm z) = \log G(\bm z) = -\sum_{k=0}^d \log(1+r_k)$, $H(\bm z)=g(\bm z)-\frac 12 \bm z^T \bm z$. Then $\nabla^2 H = \nabla^2 g - \bm I_d$. Let $\bm z_{\star}$ solves $\nabla H(\bm z) = \bm 0$, $\bm \Sigma_{\star}=(-\nabla^2 H(\bm z_{\star}))^{-1}$. We have
      \begin{equation}
      	\frac{\partial g}{\partial z_i} = -\sum_{k=1}^d \frac{1}{1+r_k}\frac{\partial r_k}{\partial z_i} = -\sum_{k=i}^d \frac{\sigma r_k}{1+r_k}.
      \end{equation}
      \begin{eqnarray}
      	\frac{\partial^2 g}{\partial z_iz_j} &=& -\sum_{k=i}^d \frac{\sigma[\frac{\partial r_k}{\partial z_j}(1+r_k)-\frac{\partial r_k}{\partial z_j}r_k]}{(1+r_k)^2} \nonumber \\
      	&=& -\sum_{k=i}^d \frac{\sigma \frac{\partial r_k}{\partial z_j}}{(1+r_k)^2} \nonumber \\
      	&=& -\sum_{k=i}^d \frac{\sigma^2 r_k \bm 1\{k \geq j\} }{(1+r_k)^2} \nonumber \\
      	&=& -\sum_{k=\max(i,j)}^d \frac{\sigma^2 r_k}{(1+r_k)^2}.
      \end{eqnarray}
      Denote $R_k = \frac{\sigma^2 r_k}{(1+r_k)^2}$ for simplicity. We obtain
      \begin{equation}\label{eq:nabla2g}
        \nabla^2 g = 
        \begin{pmatrix}
        	-\sum_{k=1}^d R_k & -\sum_{k=2}^d R_k & -\sum_{k=3}^d R_k & \cdots & -R_d \\
        	-\sum_{k=2}^d R_k & -\sum_{k=2}^d R_k & -\sum_{k=3}^d R_k & \cdots & -R_d \\
        	-\sum_{k=3}^d R_k & -\sum_{k=3}^d R_k & -\sum_{k=3}^d R_k & \cdots & -R_d \\
        	\vdots & \vdots & \vdots & \ddots & \vdots \\
        	-R_d & -R_d & -R_d & \cdots & -R_d         	
        \end{pmatrix}
      \end{equation}
      This matrix is negative definite, then $-\nabla^2 g = \bm H \bm H^T$, where one choice of decomposition matrix $\bm H$ is upper triangular
      \begin{equation}
      	\bm H = 
      	\begin{pmatrix}
            \sqrt{R_1} & \sqrt{R_2} & \sqrt{R_3} & \cdots & \sqrt{R_d} \\
            0 & \sqrt{R_2} & \sqrt{R_3} & \cdots & \sqrt{R_d} \\
            0 & 0 & \sqrt{R_3} & \cdots & \sqrt{R_d} \\
            \vdots & \vdots & \vdots & \ddots & \vdots \\
            0 & 0 & 0 & \cdots & \sqrt{R_d}
        \end{pmatrix}
      \end{equation}
      Denote $r_{\star}=\min_{k \in 1:d} r_k(\bm z_{\star})$. We still take the weight function $\psi(\bm x)= \exp (-\frac{\bm x^2}{2\alpha^2})$ to establish the RKHS with $\alpha^2 > 2$.
      \begin{thm}\label{thm:RBmodel}
        The CBC algorithm attains $O(N^{-1+1/\alpha^2+\delta})$ error bound for $\alpha^2 < \frac{1}{\kappa}\sqrt{\frac{6}{d^3+2d^2+2d+1}}$ in the Randleman-Bartter model, where $\kappa = \frac{\sigma^2}{2+r_{\star}+1/r_{\star}}$.
        \begin{proof}
          Clearly Assumption~\ref{assum:originalboundary} is satified. Since $\nabla^2 g$ is negative definite, it suffices to check Assumption~\ref{assum:eigenvalue} is valid.      
          
          According to previous theorem, we need to compute the maximal eigenvalue of $\bm \Sigma_{\star}$, and the minimal eigenvalue of $\nabla^2 g(\bm z_{\star})$
          which is unlikely to work out in theoretical analysis. Instead, estimating the eigenvalue by inequalities is performed in the following parts.  
                  
          Note that $-\nabla^2 g$ is a normal matrix. Hence we use a classic conclusion in linear algebra:
          \begin{equation}\label{eq:normaleq}
          	\sum_{k=1}^d \nu_k^2 = ||-\nabla^2 g||_F^2,
          \end{equation}
          where $\nu_k'$s denote all eigenvalues of $-\nabla^2 g$, and $||\cdot||_F$ denote the Frobenius norm, which is equivalent to the nonnegative square root of square sum of all entries. Since $R_k = \frac{\sigma^2 r_k}{(1+r_k)^2} \geq \kappa$, we estimate the Frobenius norm as follows:
          \begin{eqnarray}\label{eq:estofFnorm}
          	||-\nabla^2 g||_F^2 &=& \sum_{j=1}^d (2j-1)(\sum_{k=j}^d R_k)^2 \nonumber \\
         	&\geq& \sum_{j=1}^d (2j-1)(d-j+1)^2\kappa^2 \nonumber \\
          	&=& d\frac{(d^3+2d^2+2d+1)\kappa^2}{6}) :=S_d
          \end{eqnarray}
          Thus we have
          \begin{equation}
          	\max_{k \in 1:d} \nu_k \geq \sqrt{S_d/d}.
          \end{equation}
          Finally
          \begin{equation}
          	h = -\max_{k \in 1:d} \nu_k \leq -\sqrt{\frac{(d^3+2d^2+2d+1)}{6}}\kappa.
          \end{equation}
          For $\max_{k \in 1:d} \lambda_k$, we just use the trivial upper bound 1. We complete the proof by Theorem~\ref{thm:alternative}.
        \end{proof}          
      \end{thm}
      One may worry about the upper bound of $\alpha$ in Theorem~\ref{thm:RBmodel}, which decreases as the dimension $d$ grows, eventually leading $\alpha > 2$ intractable. This problem is not crucial. Firstly, the estimation of $\max_{k \in 1:d} \lambda_k$ is naive, thus improvement still remains. Secondly, if we take $\sigma=0.1$ for example, $\kappa \leq \sigma^2/4=0.025$, making $d=61$ still valid for the model, which is enough for most practical cases.
      
      We report in Figures~\ref{Figure1} and \ref{Figure2} the numerical results for dimension $d=5$ and $d=16$, respectively. The coefficients used are $r_0=0.1$, $\sigma=0.01$, $\gamma_u = \Big((|u|!)^2\prod_{i \in u} \frac{\tilde{\kappa}}{i^{\eta}}\Big)^{\frac{1}{1+\lambda}}
      $, where $\tilde{\kappa}=0.1$, $\eta=3.1$, $\lambda=0.51$ (see \cite{Graham2015}).
      
      In the setting of MC, ODIS and LapIS are more effective than not applying IS (i.e., NONE), supporting the benefits of using the two IS methods. Moreover, all of them have RMSEs decaying approximately at the canonical MC rate $O(N^{-1/2})$ as the sample size $N$ increases. The situation remains similar in the setting of RQMC. Both ODIS and LapIS converge faster than MC setting, and a nearly $O(N^{-1})$ error rate can be observed, which comforms to theoretical analysis above. As the dimension increases, although tiny fluctuation occurs, the main results keep the same. In conclusion, performing IS is efficient in this example, especially combining LapIS with randomly shifted lattice rule.
      
      \begin{figure}
      	\centering
      	\begin{center}
      		\includegraphics[height=160pt,width=240pt]{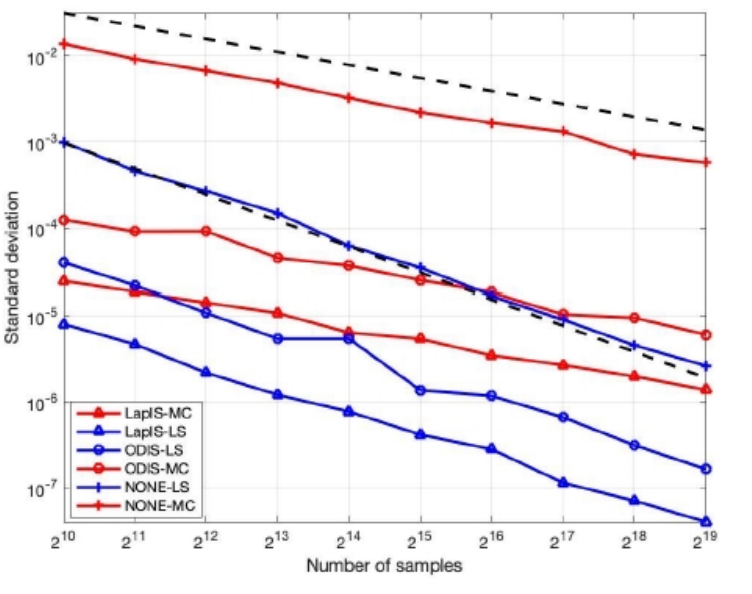}
      	\end{center}
      	\caption*{Figure1: RMSE for Randleman-Bartter short-term interest model with $d=5$.}
      	\label{Figure1}
      \end{figure}
      
      \begin{figure}
      	\centering
      	\includegraphics[height=180pt,width=270pt]{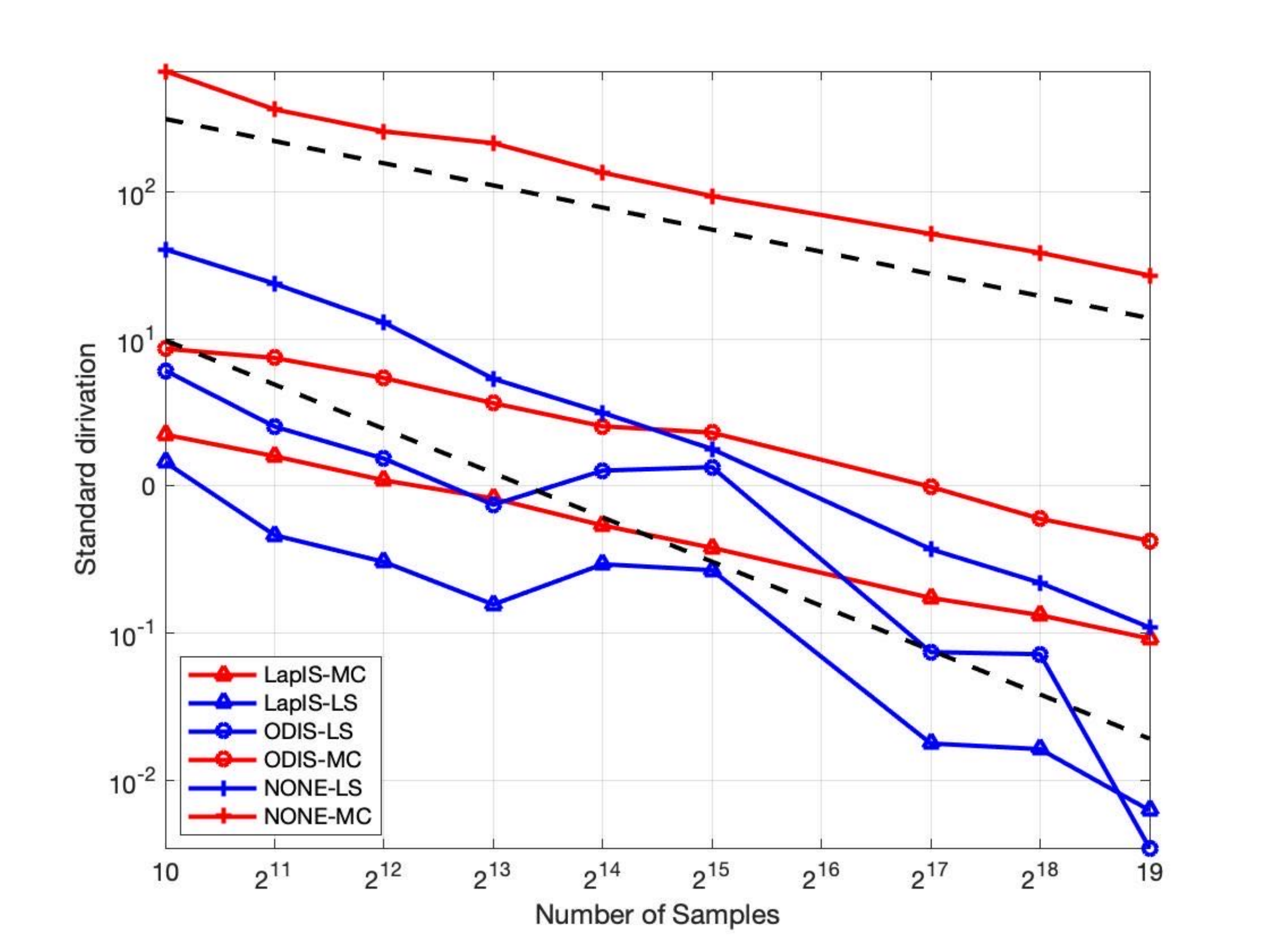}
      	\caption*{Figure2: RMSE for Randleman-Bartter short-term interest model with $d=16$.}
      	\label{Figure2}
      \end{figure}
      
	\section{Conclusion}\label{sec:con}
      This research delves into augmenting the convergence rate of the Monte Carlo (MC) method, denoted as by amalgamating quasi-Monte Carlo (QMC) techniques and importance sampling (IS). QMC, a deterministic counterpart of MC, offers an error bound of $O(\frac{(logN)^d}{N})$ for d-dimensional integrals. The adoption of randomized QMC (RQMC) enhances computational efficacy. The study underscores the synergy of IS with randomly shifted lattice rules and the intricacies in deducing a theoretical convergence rate for IS within QMC frameworks. Given the ubiquity of Gaussian measure integrals in finance and statistics, the study evaluates optimal drift importance sampling (ODIS) and Laplace importance sampling (LapIS) as significance densities. Preliminary results indicate that the IS-randomly shifted lattice rule can potentially achieve an $O(N^{-1})$. 
 error bound under certain conditions.
	
	\section{Appendix}
	In this appendix, we deduce the simple form of $\hat{\theta}(h)$. Firstly,
	\begin{eqnarray}\label{eq:invfourierbefore1}
		\hat{\theta}(h) &=& \int_0^1 \int_{\Phi^{-1}(x)}^0 \frac{\Phi(t)-x}{\psi^2(t)} \exp(-2\pi ihx)dtdx+ \int_0^1 \int_{\Phi^{-1}(1-x)}^0 \frac{\Phi(t)-1+x}{\psi^2(t)} \exp(-2\pi ihx)dtdx \nonumber \\
		&=& \int_0^1 \int_{\Phi^{-1}(x)}^0 \frac{\Phi(t)-x}{\psi^2(t)} \exp(-2\pi ihx)dtdx+ \int_0^1 \int_{\Phi^{-1}(x)}^0 \frac{\Phi(t)-x}{\psi^2(t)} \exp(2\pi ihx)dtdx \nonumber \\
		&=& 2\int_0^1 \int_{\Phi^{-1}(x)}^0 \frac{\Phi(t)-x}{\psi^2(t)} \cos(2\pi hx)dtdx. \nonumber
	\end{eqnarray}
	Split the outer integral at $\Phi(0)$, by Fubini theorem
	\begin{eqnarray}\label{eq:invfourierafter1}
		\hat{\theta}(h) &=& 2\int_0^{\Phi(0)} \int_{\Phi^{-1}(x)}^0 \frac{\Phi(t)-x}{\psi^2(t)} \cos(2\pi hx)dtdx + 2\int_{\Phi(0)}^1 \int_{\Phi^{-1}(x)}^0 \frac{\Phi(t)-x}{\psi^2(t)} \cos(2\pi hx)dtdx \nonumber \\
		&=& 2\int_{-\infty}^0 \frac{1}{\psi^2(t)}\int_0^{\Phi(t)}(\Phi(t)-x)\cos(2\pi hx)dxdt+2\int_0^{\infty} \frac{1}{\psi^2(t)}\int_{\Phi(t)}^1(x-\Phi(t))\cos(2\pi hx)dxdt \nonumber \\
		&=& 2\int_{-\infty}^0 \frac{1}{\psi^2(t)} \frac{\sin^2(\pi h\Phi(t))}{2\pi^2h^2}+2\int_0^{\infty} \frac{1}{\psi^2(t)} \frac{\sin^2(\pi h\Phi(t))}{2\pi^2h^2} dt \nonumber \\
		&=& \frac{1}{\pi^2h^2} \int_{-\infty}^{\infty} \frac{\sin^2(\pi h\Phi(t))}{\psi^2(t)} dt \nonumber \\
		&=& \frac{1}{\pi^2h^2} \int_0^1 \frac{\sin^2(\pi ht)}{\psi^2(\Phi^{-1}(t))\phi(\Phi^{-1}(t))} dt. \nonumber
	\end{eqnarray}
	\bibliographystyle{siamplain}
	\bibliography{references}
\end{document}